\RequirePackage{fix-cm}
\documentclass[smallextended]{svjour3}

\smartqed  
\usepackage{graphicx}
\usepackage{hyperref}
\usepackage{amssymb}
\usepackage{mathtools}      
\usepackage{amssymb,amsmath}
\usepackage{color}
\usepackage{caption}
\usepackage{subcaption}
\newtheorem{mydef}{Definition}[section]

\begin{document}

\title{Weighted Extended B-Spline  Finite Element Analysis of a coupled  system of general  Elliptic equations 
}

\titlerunning{Coupled elliptic problem}        

\author{Ayan Chakraborty  \and
         BV. Rathish Kumar 
}


\institute{Rathish Kumar \at
              Faculty Building,IIT Kanpur \\
               \email{bvrk@iitk.ac.in}           
           \and
           Ayan Chakraborty \at
              \email{ayancha@iitk.ac.in}
}

\date{Received: date / Accepted: date}

\maketitle

\begin{abstract}
In this study we establish the existence and uniqueness of the solution of a coupled system of general elliptic equations with anisotropic diffusion , non-uniform advection and variably influencing reaction terms  on Lipschitz continuous domain $\Omega \subset \mathbb{R}^m $ (m$\geq$1) with a Dirichlet boundary. Later we consider the finite element (FE) approximation of the coupled equations in a meshless framework based on weighted extended B-Spine functions (WEBS).The a priori error estimates corresponding to the finite element analysis are derived to establish the convergence of the corresponding FE scheme and the numerical methodology has been tested on few examples.

\keywords{Finite element \and Coupled Elliptic \and Existence and Uniqueness \and Error estimates}
\end{abstract}

\section{Introduction}

In this paper we are considering the following class of coupled system,namely general elliptic equations, related to the advection-reaction-diffusion systems appearing in the chemical and biological phenomena:
\begin{align}
\left\{
	\begin{array}{ll}
 - \nabla.(\mathfrak{P}(\nabla u_1)')+\mathfrak{R}.\nabla u_1+q_1 u_1+q_2 u_2 = f_1 ,~~ in ~ \Omega \\
-\nabla.(\mathfrak{P}(\nabla u_2)')+\mathfrak{R}.\nabla u_2+q_1 u_2-q_2 u_1 = f_2 ,~~ in ~ \Omega
\end{array}
\right.
\label{eq1}
  \end{align}

subjected to the Dirichlet boundary condition , where $\Omega$ is a Lipschitz continuous bounded  domain in $\mathbb{R}^m$ and the functions $q_1$ , $q_2$ are real essentially bounded which, in order to simplify the exposition,are assumed through out this paper to satisfy certain conditions (see section \ref{sec2}).$\mathfrak{P}$ and $\mathfrak{R}$ are  $m\times m$ symmetric matrix and $m\times1$ vector of essential bounded functions respectively.The elliptic system (\ref{eq1}) represents a steady case of reaction-diffusion system of interest in mathematical biology and physics\cite{sweers2003bifurcation}.Here real functions $u_1,u_2$ which are called activator and inhibitor respectively, can be interpreted as relative concentrations of two substances known as morphogens and the functions $\Phi_1,\Phi_2$ models autocatalytic and saturation effects.In the recent years,a lot of research has been focused  on the reaction-diffusion system (\ref{eq1}) with cross-diffusion  ,electrochemical engineering problem, from theoretical and numerical aspects, and among them coupled system of elliptic equations have received considerable attention,various forms of this system have been proposed in the literature. The aim of this paper is to obtain an existence and uniqueness of the solutions on any bounded Lipschitz  domain, in addition to carry out  a  new meshless (WEB-S) numerical method for the approximate solutions of the system. \\
On an arbitrary domain finite element approximation of these equations is a highly sought after approach to obtain their numerical solution;In particular WEB-FEM combines the computational advantages of B-Splines and standard mesh-based finite elements.Further it attains the degree and smoothness to be chosen flexibly without substantially increasing the size of problem.Off late spectral Galerkin method \cite{chen2015new}, Lattice-Boltzmann schemes \cite{suga2009stability} etc  have been used  for numerical approximations.In \cite{boglaev2012numerical} Boglaev  have used method of upper and lower solutions , and construct monotone sequence for difference scheme to approximate the solution of coupled system and Xiu et.al \cite{chen2015local} used stochastic Galerkin and stochastic collocation method in conjunction with the gPC expansions. In view of the computational advantages WEBS-FEA is one of highly  desired approach to solve the coupled elliptic system.In the current literature  no work is reported on WEBS-FEA of the general coupled elliptic problem.\\

Weighted Extended B-splines are a new class of finite element  basis functions on a conventional cartesian grid of nearly zero cost for solving Dirichlet problems on bounded domains in arbitrary dimensions.It was first proposed by Hollig et al. in \cite{hollig2001weighted}-\cite{hollig2003nonuniform}. The WEB-FEM  does not require any grid generation. Because of the simple subdivision formulas for tensor product B-splines, grid refinement techniques are easily implemented. The WEB-FEM is a meshless  technique which combines the computational efficiency of B-splines and standard mesh-based elements,with weight functions taking care of the boundary.\\

 Here the finite elements are constructed with scaled translates \textit{$b^n_k$,k} $\in$ $\mathbb{Z}^m$ of the standard m-variate tensor product B-spline of order n. B-splines \textit{$b^n_k$} (\textit{$b_k$ in abbrev}) are polynomials of degree n-1 in the variables $x_1,x_2,x_3. . . . .x_m$ on each grid cells $\mathbb{Q}_l$ = h ($[0,1]^m$ + $l$), $k_i \leq l_i\leq k_i + n$ in their support(see figure (\ref{1}) ; provided by H$\ddot{o}$llig).In case of Dirichlet problem on a bounded domain $ \Omega \subset \mathbb{R}^m$, we multiply \textit{$b_k$} by a smoothed version \textit{w} of the distance function to $\partial \Omega$. Then the span of weighted B-splines
\begin{align*}
span \{w b_k :  \Omega \cap supp~ b_k \neq \phi \}
\end{align*}
is a possible finite element subspace which conforms both the boundary conditions and yields approximations of optimal order.\\

\textit{ In our forthcoming discussion bold letters represents a vector field (i.e, vector valued functions). Henceforth , the grid width used for the spline approximation is denoted by h, and for the functions f,g we write f $\preceq$ g , if f $\leq$ cg and f $\asymp$ g ,if f=cg  for some positive constant c which doesn't depend on the grid width , indices, or arguments of functions and if $\mathbf{x} = (x_1,x_2) ,  \mathbf{y} = (y_1,y_2) $ then $\mathbf{x} \star \mathbf{y} = (x_1 y_1 , x_2 y_2)$ }\\

The paper is organized under six sections. Abstract variational formulation of the problem is introduced in section (\ref{sec1}). Section (\ref{sec2}) deals with the  existence and uniqueness aspects of the problem. Section (\ref{sec4}) deals with  a priori estimates and the convergence analysis. Numerical examples are presented in section (\ref{sec5})

\begin{figure}[h!]
  \includegraphics[width=0.75\textwidth]{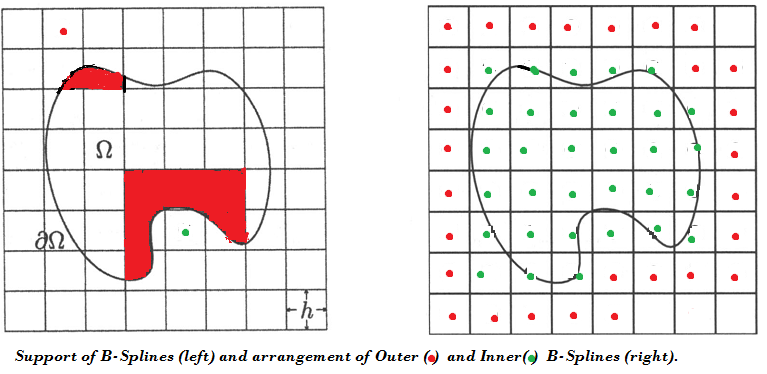}
\caption{Outer and Inner b-splines}
\label{1}       
\end{figure}

\section{Abstract Variational Formulation}
\label{sec1}
Let us consider an abstract boundary value problem
\begin{align*}
\mathcal{L}\mathbf{u} = \mathbf{f} ,~~ in~ \Omega ~~ and \mathcal{B}\mathbf{u} = 0 ~on ~ \partial \Omega
\end{align*}
with a differential operator $\mathcal{L}$ and an operator $\mathcal{B}$ describing the boundary conditions. Moreover, this problem admits a variational formulation ,(see \cite{brenner2008mathematical})
\begin{align}
a(\mathbf{u,v})= \lambda (\mathbf{v}) ,~~ \mathbf{v}~ \in H
\label{eq2}
\end{align}
where $a$ is a bilinear form and $\lambda$ is a linear functional on a Hilbert Space $H$.Then, for a finite element subspace $\mathbb{B}_h \subset H $, the Ritz-Galerkin approximation $\mathbf{u}_h$ is defined by
\begin{align*}
a(\mathbf{u}_h,\mathbf{v}_h) = \lambda(\mathbf{v}_h) ,~ \mathbf{v}_h ~ \in ~ \mathbb{B}_h
\end{align*}
The coefficients $\mathbf{a}_i$ of $\mathbf{u}_h$  with respect to a basis $\mathbf{B}_i$ of $\mathbb{B}_h$ , are determined via the linear system obtained by using $\mathbf{v}_h$ = $\mathbf{B}_k$ as test functions.

\subsection{Ritz-Galerkin Approximation}
The Ritz-Galerkin approximation $\mathbf{u}_h (\sum_i \mathbf{a}_i \star \mathbf{B}_i$ $\in$ $\mathbb{B}_h$ $\subset H $) of the variational problem (\ref{eq2}) is determined by the linear system
\begin{align}
&\sum_i a(\mathbf{B_i,B_k}) = \lambda(\mathbf{B}_k) \nonumber\\
\textrm{which is written as},\nonumber\\
&\mathbf{G}_h \mathbf{U=F}.\nonumber
\nonumber
\end{align}

 It is well known (cf. \cite{gilbarg2015elliptic} and \cite{strang1973analysis}) that the error of the approximation $\mathbf{u}_h$ can be bounded in terms of the distance of u to the subspace $\mathbb{B}_h$, spanned by $\mathbf{B}_i$.For  $\mathbf{v}$ $\in$ $H^l(\Omega)$
\begin{align*}
||\mathbf{v}||_{l,\Omega} = \big(\sum_{|\alpha|\leq l} \int_{\Omega}|D^{\alpha} \mathbf{v}|^2 \big)^{1/2}
\end{align*}
represents the norm of the Sobolev space $H^l(\Omega)$  we have
\begin{align*}
||\mathbf{u-u_h}||_1 \preceq  inf_{\mathbf{v}_h \in \mathbb{B}_h }||\mathbf{u-v_h}||_1
\end{align*}

which in other words implies the following typical error estimate for the standard finite element subspaces involving piecewise polynomials of degree $\leq$ $ n- 1$
\begin{align}
||\mathbf{u-u_h}||_1 \preceq h^{n-1}||\mathbf{u}||_n
\label{eq3}
\end{align}

Finally , a feasible estimation for the Galerkin matrix is essential and  for standard finite element subspaces using quasi-uniform partitions, the condition with respect to the 2-norm can be bounded in terms of the grid width,
\begin{align}
cond_2 G_h \preceq h^{-2}
\label{eq4}
\end{align}
which is moderate enough so that iterative methods, like the preconditioned conjugate gradient algorithm, can be employed to solve the Galerkin
system efficiently and in a stable way.We shall show in the following that the properties (\ref{eq3}) and (\ref{eq4}) remain valid for our new class of finite elements. Hence,WEB-FEM conforms the basic requirements of approximation.

\section{Coupled System of General Elliptic Operators}
\label{sec2}
In this section we will study some questions related to the existence and uniqueness of solutions , of the coupled systems of equations of general elliptic operators.

\subsection{Existence and uniqueness}

Consider the following system of general Elliptic Equations :
 \begin{align*}
 \left\{
	\begin{array}{ll}
 - \nabla.(\mathfrak{P}(\nabla u_1)')+\mathfrak{R}.\nabla u_1+q_1(\mathbf{x}) u_1+q_2(\mathbf{x}) u_2 = f_1(\mathbf{x}) ,~~ in ~ \Omega\\
 -\nabla.(\mathfrak{P}(\nabla u_2)')+\mathfrak{R}.\nabla u_2+q_1(\mathbf{x}) u_2-q_2(\mathbf{x}) u_1 = f_2(\mathbf{x}) ,~~ in ~ \Omega
 \end{array}
 \right.
\end{align*}

with $u_1=u_2$ = 0 on the boundary, $\mathfrak{P}=(a_{ij}(\mathbf{x}))_{n\times n},$  is a symmetric matrix and $ \mathfrak{R}=(r_k(\mathbf{x}))_{n\times 1}~~ ,~' $ denote transpose. If $f_1,f_2,u_1,u_2 \in L^2(\Omega)$ then from regularity theorem $u_1,u_2~ \in H^2(\Omega)$.Precisely,\textbf{u} $  \in  (H^2(\Omega)\bigcap H^1_0(\Omega))^2$ \\

Let $\Omega$ be a bounded open set in $\mathbb{R}^n$. Consider the space $(H_0^1(\Omega))^2=V$ with its product norm($ ||(._1,._2)||=||._1|| + ||._2||$). Define for \textbf{u}=($u_1,u_2$), \textbf{v}=($v_1,v_2$) in V, the bilinear form \\

a(\textbf{u},\textbf{v})= $\int_{\Omega} [\nabla v_1 \mathfrak{P}(\nabla u_1)'+\nabla v_2 \mathfrak{P}(\nabla u_2)']+ \int_{\Omega}[v_1\mathfrak{R}.\nabla u_1+v_2\mathfrak{R}.\nabla u_2] + \int_{\Omega} \textbf{u}\mathfrak{Q}\textbf{v}'$\\
\[where,\mathfrak{Q}=\left(\begin{array}{cc}
 q_1 & -q_2  \\
 q_2 & q_1  \end{array} \right)
\]\\
Clearly,
\begin{align}
 a(\textbf{u},\textbf{u})=\int_{\Omega} [\nabla u_1 \mathfrak{P}(\nabla u_1)'+\nabla u_2 \mathfrak{P}(\nabla u_2)']+ \int_{\Omega}[u_1\mathfrak{R}.\nabla u_1+u_2\mathfrak{R}.\nabla u_2] + \int_{\Omega} q_1(u_1^2 + u_2^2)
 \label{eq5}
\end{align}

which may not be coercive in general. Thus if \textbf{f} = ($f_1,f_2$)~ $\in ~ (L^2(\Omega))^2$, we cannot use Lax-Milgram lemma directly to prove the existence of a solution  to the problem: Find \textbf{u} $\in V $ such that
\begin{align}
a(\textbf{u},\textbf{v})=\langle\textbf{f},\textbf{v}\rangle , for ~ every ~ \mathbf{v}~ \in V
\label{eq6}
\end{align}

\begin{theorem}

Let us assume that $\exists$ a $\alpha >$0 such that, $A'\mathfrak{P} A \geq \alpha A'A $   and ess-inf\{$q_1(x)$ : $x\in \Omega $\}=$ \kappa \geq \frac{\alpha}{2} + \frac{\mathcal{B}^2}{2 \alpha}$ a.e in $\Omega$ where, $\mathcal{B}^2 = \Sigma_1^n \parallel r_k \parallel ^2 _\infty  ~~   a_{ij},r_k,q_1,q_2 \in L^\infty(\Omega)$ then (\ref{eq6}) possess an unique solution.
\label{thm1}

\end{theorem}

\begin{proof}
 We shall proceed as follows. From (\ref{eq5}) we obtain,
\begin{align*}
a(\textbf{u,u}) \succeq (|u_1|^2_{H^1(\Omega)} +|u_2|^2_{H^1(\Omega)}) +  \int_{\Omega}\sum_k r_k(x)[\frac{\partial u_1}{\partial x_k}u_1+\frac{\partial u_2}{\partial x_k}u_2] +\int_{\Omega} q_1(u_1^2 + u_2^2)
\end{align*}

Estimation from H$\ddot{o}$lder's Inequality,

\begin{eqnarray}\label{eqexpmuts}
|\int_{\Omega}\sum_k r_k(x)\frac{\partial u_1}{\partial x_k}u_1|
\nonumber\\
& \leq &
\int_{\Omega} \Sigma_k |r_k(x)| |\frac{\partial u_1(x)}{\partial x_k}| |u_1|
\nonumber\\
&\leq &
 \Sigma_k||r_k|| _{L^\infty (\Omega)}\int_{\Omega} |\frac{\partial u_1(x)}{\partial x_k}| |u_1|
\nonumber\\
& \leq &
\Sigma_k||r_k|| _{L^\infty (\Omega)} ||\frac{\partial u_1}{\partial x_k}||_{L^2(\Omega)} ||u_1||_{L^2(\Omega)}
\nonumber\\
& \leq &
\mathcal{B} |u_1|_{H^1(\Omega)} ||u_1||_{L^2(\Omega)}
\end{eqnarray}
Similarly,
\begin{align*}
|\int_{\Omega}\sum_k r_k(x)\frac{\partial u_2}{\partial x_k}u_2| \leq \mathcal{B} |u_2|_{H^1(\Omega)} ||u_2||_{L^2(\Omega)}
\end{align*}
Therefore,
\begin{eqnarray}
 a(\textbf{u,u}) \succeq (|u_1|^2_{H^1(\Omega)} +|u_2|^2_{H^1(\Omega)}) - \mathcal{B} (|u_1|_{H^1(\Omega)} ||u_1||_{L^2(\Omega)} + |u_2|_{H^1(\Omega)}
 \nonumber\\
 + ||u_2||_{L^2(\Omega)}) + \kappa (||u_1 ||^2_{L^2(\Omega)}+ ||u_2||^2_{L^2(\Omega)})
\end{eqnarray}

 from the inequality, $ ab\leq \frac{\delta a^2}{2}+ \frac{b^2}{2\delta}$  finally we obtain ,
\begin{align}
 a(\textbf{u,u}) \succeq ( ||u_1||_{H^1(\Omega)} ^2 + ||u_2||_{H^1(\Omega)} ^2) \succeq ||\mathbf{u}||^2_V ,~~ i.e~ coercive
\end{align}
Moreover from Poincare and C-S Inequality it can be shown that,
\begin{align}
a(\textbf{u,v})\preceq ||\mathbf{u}||_V ||\mathbf{v}||_V  ,~~ i.e ~ continuous
\end{align}
Hence,Lax-Milgram lemma ensures a unique solution of (\ref{eq6})
\end{proof}


\begin{mydef} Weighted Extended B-Splines [WEB-S]\\

For i $\in$ I, the WEB-S $B_i$ is defined by
\begin{align*}
B_i=\frac{w}{w(x_i)} \Big (b_i+\sum\limits_{j\in{J}}e_{i,j}b_j \Big) ,
\end{align*}
where $x_i$ denotes the center of the grid cell $\mathbb{Q}_{i+l(i)}$ corresponding $b_i$.The coefficients $e_{i,j}$ satisfy
\begin{align*}
|e_{i,j}| \preceq 1, e_{i,j}=0 for ||i-j||\succeq 1
\end{align*}
and are chosen so that all weighted polynomials ( abbrev wp) of order n are contained in the web space $\mathbb{B}_h=span\{B_i: i\in \mathbf{I}$\}
\label{def1}
\end{mydef}

\section{Convergence of WEB-Spline}
\label{sec4}
In this section we prove that the numerical solution [$u_{1,h},u_{2,h}$] converges to a weak solution [$u_1,u_2$] as h $\rightarrow$ 0. Moreover we have extended all our preceding results in vector field.

\subsection{Error Estiamtion}
\begin{theorem}
Assume $\mathbf{u} \in H_0^n(\Omega) $ is the solution  of (\ref{eq6}) and $\mathbf{u}_h$ is the solution of discrete system of (\ref{eq1}) , then the following error estimation holds :
\begin{align*}
|\mathbf{u-u_h}|_1 \preceq h^{n-1} ||\mathbf{u}||_n
\end{align*}
\end{theorem}

\begin{proof}

Let, $\mathbf{u_h}=(u_{1,h},u_{2,h})$ be the  discrete  and $\mathbf{u}=(u_1,u_2)$ be a weak solution of the system (\ref{eq1}). Then we have the following error equations.

\begin{align*}
\left\{
	\begin{array}{ll}
\int_{\Omega} \nabla v_1 \mathfrak{P} [\nabla(u_1 - u_{1,h})]' + \int_{\Omega} \mathfrak{R}.\nabla (u_1 - u_{1,h}) v_1 + \int_{\Omega} q_1 (u_1 - u_{1,h}) v_1 +\int_{\Omega} q_2 (u_2 - u_{2,h}) v_1 =0\\
\int_{\Omega} \nabla v_2 \mathfrak{P} [\nabla(u_2 - u_{2,h})]' +\int_{\Omega} \mathfrak{R}.\nabla (u_2 - u_{2,h}) v_2 + \int_{\Omega} q_1 (u_2 - u_{2,h}) v_2 - \int_{\Omega} q_2 (u_1 - u_{1,h}) v_2  = 0
\end{array}
\right.
\end{align*}

for all $\mathbf{v}=(v_1,v_2)$  $\in$ V \\

taking $(u_{1,h} - \mathrm{P_h}u_1 , u_{2,h} - \mathrm{P_h}u_2)$ as a test function  to the  above discrete error equation, we have the equality.

\begin{eqnarray}\label{eqexpmuts}
\int_{\Omega} \nabla (u_{1,h} - \mathrm{P_h}u_1) \mathfrak{P} [\nabla(u_1 - u_{1,h})]' + \int_{\Omega} \mathfrak{R}.\nabla (u_1 - u_{1,h})(u_{1,h} - \mathrm{P_h}u_1)
\nonumber\\
+ \int_{\Omega} q_1 (u_1 - u_{1,h}) (u_{1,h} - \mathrm{P_h}u_1) +\int_{\Omega} q_2 (u_2 - u_{2,h}) (u_{1,h} - \mathrm{P_h}u_1) = 0
\label{eq8}
\end{eqnarray}

\begin{eqnarray}\label{eqexpmuts}
\int_{\Omega} \nabla (u_{2,h} - \mathrm{P_h}u_2) \mathfrak{P} [\nabla(u_2 - u_{2,h})]' + \int_{\Omega} \mathfrak{R}.\nabla (u_2 - u_{2,h})(u_{2,h} - \mathrm{P_h}u_2)
\nonumber\\
 +\int_{\Omega} q_1 (u_2 - u_{2,h}) (u_{2,h} - \mathrm{P_h}u_2) - \int_{\Omega} q_2 (u_1 - u_{1,h}) (u_{2,h} - \mathrm{P_h}u_2)=0
 \label{eq9}
\end{eqnarray}

 from (\ref{eq8}),

\begin{eqnarray*}\label{eqexpmuts}
\int_{\Omega} \sum_{i,j} a_{ij} \frac{\partial}{\partial x_i}(u_1 - u_{1,h}) \frac{\partial}{\partial x_j}(u_1 - u_{1,h}) &=& \int_{\Omega} \sum_{i,j} a_{ij} \frac{\partial}{\partial x_i}(u_1 - u_{1,h})\frac{\partial}{\partial x_j}(u_1-\mathrm{P_h}u_1) \\ +\int_{\Omega} \mathfrak{R}.\nabla (u_1 - u_{1,h})(u_{1,h} - \mathrm{P_h}u_1) &-& \int_{\Omega} q_1 (u_1 - u_{1,h})^2  \\ +\int_{\Omega} q_1 (u_1 - u_{1,h})(u_1 - \mathrm{P_h}u_1)& -& \int_{\Omega} q_2(u_2 - u_{2,h}) (u_1 - u_{1,h}) + \int_{\Omega} q_2 (u_2 - u_{2,h})(u_1 - \mathrm{P_h}u_1)
\end{eqnarray*}

Similarly (\ref{eq9}) yields,

\begin{eqnarray*}\label{eqexpmuts}
\int_{\Omega} \sum_{i,j} a_{ij} \frac{\partial}{\partial x_i}(u_2 - u_{2,h}) \frac{\partial}{\partial x_j}(u_2 - u_{2,h}) &=& \int_{\Omega} \sum_{i,j} a_{ij} \frac{\partial}{\partial x_i}(u_2 - u_{2,h}) \frac{\partial}{\partial x_j}(u_2-\mathrm{P_h}u_2) \\+\int_{\Omega} \mathfrak{R}.\nabla (u_2 - u_{2,h})(u_{2,h} - \mathrm{P_h}u_2) &-& \int_{\Omega} q_1 (u_2 - u_{2,h})^2 \\ +  \int_{\Omega} q_1 (u_2 - u_{2,h})(u_2 - \mathrm{P_h}u_2) &+& \int_{\Omega} q_2(u_1 - u_{1,h}) (u_2 - u_{2,h}) - \int_{\Omega} q_2 (u_1 - u_{1,h})(u_2 - \mathrm{P_h}u_2)
\end{eqnarray*}

adding both,

\begin{eqnarray*}\label{eqexpmuts}
\int_{\Omega} \sum_{i,j} a_{ij} [\frac{\partial}{\partial x_i}(u_2 - u_{2,h}) \frac{\partial}{\partial x_j}(u_2 - u_{2,h}) + \frac{\partial}{\partial x_i}(u_1 - u_{1,h}) \frac{\partial}{\partial x_j}(u_1 - u_{1,h})] \\ =  \int_{\Omega} \sum_{i,j} a_{ij} [\frac{\partial}{\partial x_i}(u_2 - u_{2,h}) \frac{\partial}{\partial x_j}(u_2-\mathrm{P_h}u_2) + \frac{\partial}{\partial x_i}(u_1 - u_{1,h}) \frac{\partial}{\partial x_j}(u_1-\mathrm{P_h}u_1)] \\- \int_{\Omega} q_2 [(u_1 - u_{1,h})(u_2 - \mathrm{P_h}u_2) - (u_2 - u_{2,h})(u_1 - \mathrm{P_h}u_1)] - \int_{\Omega} q_1 [(u_1 - u_{1,h})^2 +(u_2 - u_{2,h})^2] \\+ \int_{\Omega} q_1 [(u_1 - u_{1,h})(u_1 - \mathrm{P_h}u_1) + (u_2 - u_{2,h})(u_2 - \mathrm{P_h}u_2)]
\end{eqnarray*}

from our assumption , $\mathfrak{P}$ is positive definite by $\alpha ~>$ 0 and $a_{ij}~\in~L^{\infty}$ yields ,

\begin{eqnarray*}\label{eqexpmuts}
\alpha \int_{\Omega} \sum_i  [|\frac{\partial}{\partial x_i}(u_2 - u_{2,h})|^2 + |\frac{\partial}{\partial x_i}(u_1 - u_{1,h})|^2] \preceq |u_1 - u_{1,h}|_1 |u_1 - \mathrm{P_h}u_1|_1  + |u_2 - u_{2,h}|_1 |u_2 - \mathrm{P_h}u_2|_1\\ + |u_1 - u_{1,h}|_0 (|u_2 - \mathrm{P_h}u_2|_0 + |u_1 - \mathrm{P_h}u_1|_0) + |u_2 - u_{2,h}|_0 (|u_1 - \mathrm{P_h}u_1|_0\\ + |u_2 - \mathrm{P_h}u_2|_0) - \int_{\Omega} q_1 [(u_1 - u_{1,h})^2 +(u_2 - u_{2,h})^2]
\end{eqnarray*}

\begin{align}
\int_{\Omega} q_1 [(u_1 - u_{1,h})^2 +(u_2 - u_{2,h})^2]  + \frac{\alpha}{2} |\mathbf{u - u_h}|^2_1 \preceq |\mathbf{u - u_h}|_1 |\mathbf{u - \mathrm{P_h}u}|_1 + |\mathbf{u - u_h}|_0 |\mathbf{u - \mathrm{P_h}u}|_0
\label{eq10}
 \end{align}

 from the inequality , $\frac{(a+b)^2}{2} \leq (a^2 + b^2)$ and our definition of product norms.\\

Clearly $q_1 >0 $ , hence

\begin{align}
 |\mathbf{u - u_h}|_1 \preceq |\mathbf{u - \mathrm{P_h}u}|_1  ; \textrm{from Poincare Inequality}
 \end{align}

Now we give the projection error estimate for \textbf{u},
\begin{theorem} Let, $\mathbf{u}\in H^n$ be a weak solution of .Then,
\begin{eqnarray}
\|\mathbf{u-\mathrm{P_h}u}\|_1\preceq h^{n-1} \|\mathbf{u}\|_n
\label{eq11}
\end{eqnarray}
\end{theorem}

The proof uses standard quasi interpolation techniques and can be found on \cite{hollig2003finite}.

Hence from (\ref{eq10}) and (\ref{eq11}) we obtain
\begin{eqnarray*}\label{eqexpmuts}
|\mathbf{u - u_h}|_1 \preceq h^{n-1} ||\mathbf{u}||_n ,~~ if ~ \textbf{u} ~ \in ~ H^n_0
\end{eqnarray*}
required error estimation of the system (\ref{sec1}).
\end{proof}


\section{Numerical Experiments}
\label{sec5}

In this section we demonstrate a  test example belonging to the class of convection-diffusion equations. In the non isothermal chemical reaction process involving chemical species, the chemical concentrations and the temperature are governed by a coupled system of reaction diffusion equations of the form of (\ref{eq1}).
We present a selected numerical examples concerning our discussions on preceding sections.The domain in our examples is a quadrant circle $\Omega$ = $\{(x,y) : x,y \geq$0 and $ 1-x^2-y^2 \leq 0\}$, and zero boundary conditions are imposed on $\partial\Omega$.Here, n= degree of the polynomial , h= grid width , e = $L_2$ error. Corresponding  graphs for the solutions , residual error functions (for h=0.1 and highest order of n) and computation time shown in Figures were measured on a Intel Core i7-4770S CPU  3.10 GHz.\\

For  the following problem, the finite element approximation with B-splines requires :

\begin{itemize}

\item[$\bullet$] specification of the functions and constants appearing in the partial differential
equations.

\item[$\bullet$] description of the domain and the essential boundary.

\item[$\bullet$] choice of the spline space.

\end{itemize}

The domain and the essential boundary are represented by weight functions.

\subsection{Weight Functions} The domain $\Omega$ is a subset of the unit square described implicitly by a weight function $w_{\Omega}$ :\\
\begin{align*}
 \Omega = \{ (x,y) \in (0,1)^2 : w_{\Omega}(x,y) > 0\}
\end{align*}

The boundary condition u = 0 on the essential part $\Gamma$ of $\partial \Omega$ is incorporated by a
weight function w which is of one sign on $\Omega$ and vanishes linearly on $\Gamma$ cf \cite{hollig2004finite}
\begin{align*}
\Gamma = \{ (x,y) \in \partial \Omega : w_{\Omega}(x,y) = 0\}
\end{align*}

\subsection{Splines} For all boundary value problems, solutions are approximated by linear combinations of weighted B-splines:

\begin{align*}
\mathbf{u(x,y)} \approx \mathbf{u}_h(x,y) = \sum _{k_1 =1} ^ {H+n} \sum _{k_2 =1} ^ {H+n} \ w(x,y) \mathbf{u}_k (x,y) \star \mathbf{b}_k (x,y) ~ ; \hspace {2mm} k =(k_1,k_2)
\end{align*}

for (x,y) $\in (0,1)^2$  with $w_{\Omega} >$ 0.  Here

\begin{itemize}
  \item[$\bullet$] h is the grid width.

  \item[$\bullet$]  n is the degree of the B-Splines.

  \item[$\bullet$]  the coefficients $\mathbf{u}_k$ are vectors

  \item[$\bullet$] $b_k$ is the uniform tensor product B-Spline with support ,

  \begin{align*}
  [(k_1 -n-1)h,k_1 h] \times [(k_2-n-1)h,k_2 h]
  \end{align*}
  corresponds to the grid position k = $(k_1,k_2)$
\end{itemize}

\subsection{Residual for the partial differential equation} The residual of the Ritz-Galerkin approximation $\mathbf{u}_h$ is defined as
\begin{align*}
r(x,y)= (L\mathbf{u}_h)(x,y) - f(x,y)
\end{align*}

where , L is the differential operator of the boundary value problem. The relative error

\begin{align*}
e= ||r||_{0,\Omega}/||f||_{0,\Omega}
\end{align*}

with $||.||_{0,\Omega}$  denoting the $L_2$ norm on $\Omega$ provides a measure of accuracy for the solution
without having to resort to grid refinement,for details \cite{hollig2003nonuniform}

\subsection{Example: BVP with polynomial coefficients}

\begin{equation}
 - \nabla.((1+x)(1+y)\nabla u_1)+ x u_1 + xy u_2 = \frac{x}{10}-\frac{y}{100}
\end{equation}
\begin{equation}
 -\nabla.((1+x)(1+y)\nabla u_2) +x u_2- xy u_1 = \frac{x^2}{100}
\end{equation}


\begin{figure}[h!]
  \includegraphics[width=0.70\textwidth]{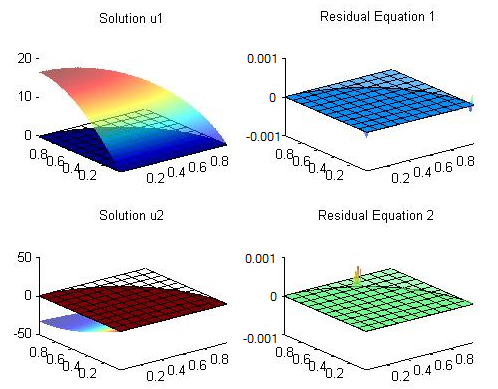}
\caption{solution functions and the residual error functions corresponding n=4 and h=0.1}
\label{sec :6.2}       
\end{figure}

\section{Conclusion}
Following  Galerkin Method of Approximation, Existence and Uniqueness of the   Non-Cooperative Elliptic Equations has been successfully established.Jackson Inequality, stability estimates of WEB-S basis and Poincare Inequality facilitate the derivation of a priori error estimates to the WEBS-FEA of this system. The proposed WEBS-FE based numerical scheme has been successfully tested on few  models.






\section*{Acknowledgements} The authors would like to express their gratitude to Dr.Klaus Hoellig and Joerg Hoerner for helping us in modifying the Matlab code.The Ph.D Fellowship of NBHM-DAE is gratefully acknowledged by first author.

\end{document}